\documentclass[11pt,reqno]{amsart}
\usepackage{amssymb}

\newtheorem{theorem}{Theorem}

\theoremstyle{remark}
\newtheorem{remark}{Remark}

\numberwithin{equation}{section}

\begin{document}

\title[Difference Equation for the Hypergeometric Function]
{Difference Equation for the Heckman-Opdam Hypergeometric Function and its confluent Whittaker limit}

\author{J.F.  van Diejen}

\address{
Instituto de Matem\'atica y F\'{\i}sica, Universidad de Talca,
Casilla 747, Talca, Chile}

\email{diejen@inst-mat.utalca.cl}

\author{E. Emsiz}

\address{
Facultad de Matem\'aticas, Pontificia Universidad Cat\'olica de Chile,
Casilla 306, Correo 22, Santiago, Chile}
\email{eemsiz@mat.puc.cl}


\thanks{This work was supported in part by the {\em Fondo Nacional de Desarrollo
Cient\'{\i}fico y Tecnol\'ogico (FONDECYT)} Grants \# 1130226 and  \# 1141114.}

\date{August 2015}

\begin{abstract}
We present an explicit difference equation for the Heckman-Opdam hypergeometric function associated with root systems. Via a confluent hypergeometric limit, an analogous difference equation is obtained for the class-one Whittaker function diagonalizing the open quantum Toda chain.
\end{abstract}

\maketitle

\section{Introduction}\label{sec1}
It is well-known that the action-angle transformations linearizing the hyperbolic Calogero-Moser model and the rational Ruijsenaars-Schneider system are inverses of each other
\cite{rui:action-angle,feh-kli:duality,pus:hyperbolic,feh-gor:duality}. At the quantum level \cite{ols-per:quantum,rui:complete,die:integrability}, this remarkable reciprocity between the two integrable particle models manifests itself as a bispectral duality in the sense of
Duistermaat and Gr\"unbaum \cite{dui-gru:differential,gru:bispectral}: viewed as a function of the spectral variable,
the eigenfunction kernel of the hyperbolic Calogero-Moser Hamiltonian solves the eigenvalue equation
for the rational Ruijsenaars-Schneider Hamiltonian
\cite{rui:finite-dimensional,cha:duality,cha:bispectrality}. In the framework of affine Hecke algebras, both quantum systems are simultaneously recovered by degeneration   from the respective centers of two affine Hecke algebras that sit inside the double affine Hecke algebra \cite{che:inverse}.  The upshot is that the Heckman-Opdam hypergeometric function diagonalizing the hyperbolic Calogero-Moser system 
\cite{hec-sch:harmonic,hec:dunkl,opd:harmonic,opd:lecture} satisfies a system of difference equations in the spectral variable stemming from the commuting quantum integrals of the
rational Ruijsenaars-Schneider system \cite{che:inverse,cha:bispectrality}. 

The present work addresses the question of computing such difference equations in explicit form for arbitrary root systems, beyond the known simplest examples associated with the (quasi-)minuscule weights. Rather than facing the challenging task of calculating the difference equations at issue from a representation of the (degenerate) double affine Hecke algebra, we exploit the well-known fact that for discrete spectral values on a (translated) cone of dominant weights the Heckman-Opdam hypergeometric function truncates in terms of Heckman-Opdam Jacobi polynomials
\cite{hec-sch:harmonic}. This permits us to derive/prove the desired difference equations in two steps: first for the discrete spectral values by performing
a $q\to 1$ degeneration of a recently found Pieri formula for the Macdonald polynomials \cite{die-ems:generalized},
and then for arbitrary spectral values upon invoking an analytic continuation argument borrowed from R\"osler \cite{ros:positive} (based on known
growth estimates for the Heckman-Opdam hypergeometric function \cite{opd:harmonic,sch:contributions,ros-koo-voi:limit} that enable one to apply Carlson's theorem \cite{tit:theory}).   For the root system of type $A$, a more direct way to recover hypergeometric difference equations straight from the Macdonald-Ruijsenaars $q$-difference operators was recently pointed out by Borodin and Gorin, who established that
 the pertinent Heckman-Opdam hypergeometric function itself can be retrieved as a limit of the Macdonald polynomials  \cite[App. B]{bor-gor:general}.

Moreover, it was shown by Shimeno \cite{shi:limit} (see also \cite{osh-shi:heckman-opdam}) that for reduced root systems of arbitrary type the class-one Whittaker function diagonalizing the open quantum Toda chain \cite{jac:fonctions,kos:quantization,has:whittaker,goo-wal:classical,sem:quantisation,bau-oco:exponential} arises in turn as a confluent hypergeometric degeneration of the Heckman-Opdam hypergeometric function. Confluences of this kind have their origin in a well-known limiting transition from the hyperbolic Calogero-Moser Hamiltonian to the Toda Hamiltonian that was observed independently by Sutherland and Ruijsenaars \cite{sut:introduction,rui:relativistic-toda} and in more general form by Inozemtsev \cite{ino:finite}.
At the level of the present difference equation, the degeneration in question manifests itself as a strong-coupling limit and gives rise to a corresponding difference equation for the Whittaker function. 
For the root system of type $A$ the resulting difference equation in the spectral variable, satisfied by the quantum Toda eigenfunction, was previously established both by means of the quantum inverse scattering method \cite{kar-leb:integral,skl:bispectrality,koz:aspects} and by means of an
integral representation for the Whittaker function \cite{bab:equations,bor-cor:macdonald}.  The  underlying bispectral  duality should again be viewed as the quantum counterpart of the remarkable fact that the respective action-angle transforms linearizing the open Toda chain and a strong-coupling limit of  rational Ruijsenaars-Schneider system are  each other's inverses \cite{rui:relativistic-toda,feh:action-angle}.

The material is organized as follows. After recalling some necessary preliminaries regarding the Heckman-Opdam hypergeometric function in Section \ref{sec2}, our difference equation is first stated for reduced root systems in Section \ref{sec3} and then proven in Section \ref{sec4}. We wrap up  in Sections \ref{sec5} and \ref{sec6}, by indicating briefly how our difference equation is to be adapted in the case of a nonreduced root system and by computing the degeneration of the difference equation pertaining to  the Toda-Whittaker level, respectively.

\section{Preliminaries}\label{sec2} 
Throughout,
the Heckman-Opdam hypergeometric function will be viewed here as a holomorphic interpolation function for the Heckman-Opdam Jacobi polynomials.
\subsection{Jacobi polynomials}
Let $W$, $Q$ and $P$ denote the Weyl group, the root lattice and the weight lattice, associated with 
an irreducible crystallographic root system $R$ spanning a real finite-dimensional Euclidean vector space $V$ with inner product $\langle \cdot ,\cdot\rangle$ \cite{bou:groupes}. Here the dual linear space will be identified with $V$ via the inner product, in particular, this means that $\alpha^\vee:=2\alpha/\langle\alpha ,\alpha\rangle\in V$ is taken as our definition of the coroot associated with $\alpha\in R$. We furthermore write $Q^+\subset Q$  for
the nonnegative semigroup generated by  a (fixed) choice of positive roots $R^+\subset R$, and  $P^+\subset P$ for the corresponding dual cone of dominant weights endowed with the dominance partial order $\leq $ (i.e. for $\mu,\lambda\in P^+$: $\mu\leq \lambda$ iff $\lambda-\mu\in Q^+$).

The normalized Heckman-Opdam Jacobi polynomials $P_\lambda$, $\lambda\in P^+$ are $W$-invariant polynomials on $V$ of the form
\cite{hec-sch:harmonic,opd:harmonic}
\begin{subequations}
\begin{equation}\label{Ja}
P_\lambda (x) = \sum_{\mu\in P^+,\, \mu \leq \lambda} c_{\lambda ,\mu} m_\mu (x) \qquad (c_{\lambda ,\mu}\in\mathbb{C})
\end{equation}
such that
\begin{equation}\label{Jb}
L P_\lambda =E (\rho_g+\lambda )P_\lambda\quad \text{with}\quad E (\xi) :=
\langle\xi ,\xi\rangle-\langle \rho_g,\rho_g\rangle ,
\end{equation}
\end{subequations}
with $P_\lambda (0)=1$. Here $m_\mu :=\sum_{\nu\in W\mu} e^\nu$ with $e^\nu (x):=e^{\langle \nu  ,x\rangle}$ ($\nu\in P$, $x\in V$), $L$ denotes the hypergeometric differential operator
\begin{equation}\label{L}
L:=\Delta + \sum_{\alpha \in R^+}  g_\alpha \Bigl( \frac{1+e^{-\alpha}}{1-e^{-\alpha}}\Bigr)\partial_\alpha 
\end{equation}
parametrized by root multiplicity parameters $g_\alpha \in (0,+\infty)$ with $g_{w\alpha}=g_\alpha$ for all $w\in W$, and
\begin{equation}
\rho_g:=\frac{1}{2}\sum_{\alpha\in R^+}g_\alpha \alpha .
\end{equation}
In $L$ \eqref{L} the operator $\Delta$ refers to the Laplacian on $V$  (i.e. $\Delta e^\nu=\langle \nu ,\nu \rangle e^\nu$) and $\partial_\alpha$ stands for the directional derivative along the root $\alpha$ (i.e. $\partial_\alpha e^\nu =\langle \nu,\alpha\rangle e^\nu$).

\begin{remark}
The leading coefficient of the Jacobi polynomial $P_\lambda(x)$ \eqref{Ja}, \eqref{Jb} normalized such that $P_\lambda (0)=1$ was computed explicitly by Opdam
\cite{opd:some,hec-sch:harmonic,opd:lecture}:
\begin{equation}
c_{\lambda,\lambda}=\prod_{\alpha\in R^+}\prod_{j=0}^{\langle \lambda, \alpha^\vee\rangle-1} \frac{\langle \rho_g, \alpha^\vee \rangle +\frac{1}{2}g_{\alpha/2}+j}{\langle \rho_g, \alpha^\vee \rangle+g_\alpha+\frac{1}{2}g_{\alpha/2}+j}
\end{equation}
(with the convention that $g_{\alpha/2}:=0$ if $\alpha/2\not\in R$). The remaining coefficients $c_{\lambda ,\mu}$, $\mu <\lambda$ can be computed in closed form via a linear recursion stemming from the eigenvalue equation \eqref{Jb} \cite[Thm. 4.4]{die-lap-mor:determinantal}.
\end{remark}

\subsection{Hypergeometric function}
Let us denote the complexification $V\otimes_\mathbb{R}\mathbb{C}$ of $V$ by $V_\mathbb{C}$, with the form $\langle \cdot ,\cdot\rangle$ and the action of $W$ extended by linearity. The Heckman-Opdam hypergeometric function
$F_\xi ( x)$ is a holomorphic  function of $\xi\in V_\mathbb{C}$ and $x\in \text{U} \subset V_\mathbb{C}$---where $\text{U}$ denotes a sufficiently small $W$-invariant tubular neighborhood
of $V$ in $V_\mathbb{C}$---satisfying the following properties \cite{hec-sch:harmonic}:
\begin{subequations}
\begin{equation}\label{Hb}
F_{\rho_g+\lambda }(x) =P_\lambda (x)\qquad (\text{for all} \ \lambda\in P^+, x\in U),
\end{equation}
\begin{equation}\label{Ha}
F_{w\xi}(x)=F_\xi (wx)=F_\xi (x)  \qquad  (\text{for all}\ \xi \in V_\mathbb{C}, x\in U, w\in W),
\end{equation}
and \cite{opd:harmonic,sch:contributions,ros-koo-voi:limit}
\begin{equation}\label{Hc}
| F_{\rho_g+\xi}(x)| \leq e^{\max_{w\in W} \text{Re} \langle w\xi ,x\rangle}
\qquad (\text{for all}\ \xi \in V_\mathbb{C}, x\in V).
\end{equation}
\end{subequations}

This hypergeometric function extends the $W$-invariant solution of the eigenvalue equation for the hypergeometric differential operator $L$ \eqref{L}
given by the normalized Jacobi polynomials, from the discrete spectral values $\xi$ in $\rho_g+P^+$ to general
values of the spectral variable  \cite{hec-sch:harmonic,hec:dunkl,opd:harmonic,opd:lecture}:
\begin{equation}\label{HDE}
L F_\xi = E (\xi ) F_\xi 
\quad\text{and}\quad F_\xi (0)=1
\qquad (\text{for all}\ \xi\in V_\mathbb{C}).
\end{equation}

\begin{remark}
It is clear from Section \ref{sec4} below that
the properties in Eqs. \eqref{Hb}--\eqref{Hc}---combined with  the analyticity in $\xi$ and $x$---in fact characterize $F_\xi (x)$ uniquely (by Carlson's theorem).
\end{remark}

\section{Hypergeometric difference equation}\label{sec3}
From now onwards except in Section \ref{sec5}, it will be assumed that $R$ is reduced unless explicitly insinuated otherwise.
To any  $\nu\in P$ we associate its stabilizer subgroup
$W_\nu:=\{ w\in W\mid w\nu =\nu\}$  generated by the root subsystem  $R_\nu:=\{\alpha\in R\mid \langle \nu ,\alpha^\vee \rangle =0\}$, and
the shortest group element  $w_\nu\in W$ such that $w_\nu (\nu)\in P^+$. 
For $\lambda\in P^+$, we denote by $P(\lambda)\subset P$ the saturated set with highest weight $\lambda$ given by
\begin{equation}
 P(\lambda ):=\bigcup_{\mu\in P^+, \mu\leq \lambda} W\mu .
\end{equation}

Our main result is the following explicit difference equation in the spectral variable for the Heckman-Opdam hypergeometric function associated with a reduced root system.
\begin{theorem}\label{HDE:thm}
For $R$ reduced and $\omega\in P^+$ small in the sense that $\langle \omega,\alpha^\vee\rangle\leq 2$ for all $\alpha\in R^+$,
the Heckman-Opdam hypergeometric function $F_\xi(x)$ satisfies the difference equation
\begin{subequations}
\begin{equation}\label{HDEa}
\sum_{\nu\in P(\omega)}
\sum_{\eta\in W_\nu(w_\nu^{-1}\omega)}
 U_{\nu,\eta }(\xi) V_{\nu }(\xi) F_{\xi+\nu }(x) =E_{\omega}(x) F_{\xi}(x)
\end{equation}
(as a holomorphic identity in $x\in U$ and $\xi\in V_\mathbb{C}$), with
\begin{eqnarray}\label{HDEb}
&& V_{\nu }(\xi ):=\prod_{\substack{\alpha\in R\\ \langle\nu,\alpha^\vee \rangle>0}} \frac{\langle \xi, \alpha^\vee \rangle+g_\alpha }{\langle \xi, \alpha^\vee \rangle}
      \prod_{\substack{\alpha\in R\\ \langle\nu, \alpha^\vee \rangle=2}} \frac{1+\langle \xi, \alpha^\vee  \rangle +g_\alpha }{1+ \langle \xi,\alpha^\vee \rangle} ,\\
&& U_{\nu,\eta } (\xi ):=\prod_{\substack{\alpha\in R_\nu \\ \langle\eta,\alpha^\vee\rangle>0}}
\frac{\langle \xi, \alpha^\vee  \rangle +g_\alpha }{\langle \xi,\alpha^\vee  \rangle}
\prod_{\substack{\alpha\in R_\nu \\ \langle\eta,\alpha^\vee\rangle=2}} 
 \frac{1+\langle \xi, \alpha^\vee \rangle -g_\alpha }{1+ \langle \xi, \alpha^\vee  \rangle} , \label{HDEc}
\end{eqnarray}
and
\begin{equation}\label{HDEd}
E_{\omega }(x):=\sum_{\mu\in P^+,\, \mu\leq \omega}  | W_\mu (\omega) | m_\mu (x)  ,
\end{equation}
\end{subequations}
where $  | W_\mu (\omega) | $ denotes the order of the orbit of $\omega$ with respect to $W_\mu$.
\end{theorem}

If $\omega$ is minuscule (i.e.  $0\leq \langle \omega,\alpha^\vee\rangle\leq 1$ for all $\alpha\in R^+$), then $P(\omega)=W\omega$ and the above difference equation becomes of the form
\begin{equation}\label{min}
\sum_{\nu\in W\omega} V_{\nu }(\xi) F_{\xi+\nu }(x) =m_{\omega}(x) F_{\xi}(x).
\end{equation}
If $\omega$ is quasi-minuscule (i.e.  $\omega\in R^+$ and $0\leq \langle \omega,\alpha^\vee\rangle\leq 1 $ for all $\alpha\in R^+\setminus \{\omega\}$), then $P(\omega)=W\omega\cup \{ 0\}$ and our difference equation reads
\begin{equation}
\sum_{\nu\in W\omega} \bigl( V_{\nu }(\xi) F_{\xi+\nu }(x)  +U_{0,\nu }(\xi) F_{\xi}(x) \bigr) =\bigl( m_{\omega}(x) +m_\omega (0) \bigr) F_{\xi}(x) ,
\end{equation}
which is readily rewritten in the form
\begin{equation}\label{qmin}
\sum_{\nu\in W\omega}  V_{\nu }(\xi) \bigl( F_{\xi+\nu }(x) -F_{\xi}(x) \bigr) =\bigl( m_{\omega}(x) -m_\omega (0) \bigr) F_{\xi}(x),
\end{equation}
with the aid of the elementary identity $\frac{1}{2}\sum_{\nu\in W\omega}  \bigl( V_{\nu }(\xi)+U_{0,\nu }(\xi) \bigr) = m_\omega (0)$.
In these simplest examples the explicit difference equation for the Heckman-Opdam hypergeometric function was first established by Cherednik for $\omega $ minuscule using the degenerate double affine Hecke algebra (cf. \cite[Eq. (4.15)]{che:inverse}), and by Chalykh for $\omega$ either minuscule or quasi-minuscule with the aid of a Baker-Akhiezer function associated with $R$ (cf. \cite[Thm. 6.9]{cha:bispectrality}).

\begin{remark}
In the case of the classical series $A_n$, $B_n$, $C_n$ and $D_n$ all fundamental weights of $R$ are small in the sense specified in Theorem \ref{HDE:thm}, whereas for the exceptional root systems
$E_6$, $E_7$, $E_8$, $F_4$ and $G_2$ the number of such small fundamental weights of $R$ is given by $5$, $4$, $2$, $2$ and $1$, respectively.
By varying $\omega$ over these small fundamental weights, the Eqs. \eqref{HDEa}--\eqref{HDEd} produce an explicit system of independent difference equations for the Heckman-Opdam hypergeometric function $F_\xi (x)$.
\end{remark}

\begin{remark}\label{rm-reg}
For technical convenience, we have assumed most of the time that our root multiplicity parameters $g_\alpha$ take positive values. It is known, however, that $F_\xi (x)$ (with $(\xi,x)\in V_\mathbb{C}\times U$) extends holomorphically in the root multiplicity parameters to complex values outside the pole locus of an overall normalization factor of the form
\cite{hec-sch:harmonic,opd:lecture}
\begin{equation}
\prod_{\alpha\in R^+}  \frac{\Gamma (\langle \rho_g,\alpha^\vee\rangle+ \frac{1}{2}g_{\alpha/2}+g_\alpha )}{\Gamma (\langle \rho_g,\alpha^\vee\rangle  +\frac{1}{2}g_{\alpha/2})} 
\end{equation}
(where $\Gamma (\cdot)$ refers to the gamma function),
which comprises in particular the closed parameter domain $\text{Re}(\frac{1}{2}g_{\alpha/2}+g_\alpha ) \geq 0$.
The explicit difference equation in Theorem \ref{HDE:thm} thus extends holomorphically to such complex root multiplicities as well.
\end{remark}

\section{Proof of Theorem \ref{HDE:thm}}\label{sec4}
Our proof of Theorem \ref{HDE:thm} exploits (and makes precise sense of) the idea that the difference operator that enters via the LHS  of our difference equation---through its action on $F_\xi (x)$ in the spectral variable---amounts to a rational degeneration of an analogous difference operator diagonalized by the Macdonald polynomials \cite[Thm. 3.1]{die-ems:generalized}. The simplest examples in Eqs. \eqref{min} and \eqref{qmin}  correspond in this respect to the rational degenerations of the celebrated Macdonald difference operators associated with the (quasi-)minuscule weights \cite{mac:orthogonal}.

\subsection{Holomorphy in the spectral variable}\label{holomorphy:sec}
The holomorphy and $W$-invariance of $F_\xi (x)$ in the spectral variable ensures that the LHS of \eqref{HDEa} is actually holomorphic for $\xi\in V_\mathbb{C}$. Indeed, the generically simple poles of the coefficients $V_\nu (\xi)$ \eqref{HDEb} and $U_{\eta ,\nu}(\xi) $ \eqref{HDEc} at root hyperplanes $\langle \xi, \alpha^\vee\rangle=0$
and at the affine root hyperplanes $\langle \xi, \alpha^\vee\rangle+1=0$ ($\alpha\in R$) cancel out in the final expression on the LHS of Eq. \eqref{HDEa}  due to the Weyl group symmetry. For the root hyperplanes this is obvious (from the reflection-symmetry) and for the affine root hyperplanes this follows from a detailed residue analysis performed (at the trigonometric level) in the appendix of Ref. \cite{die-ems:generalized}.

\subsection{Pieri formula}
At the spectral value $\xi=\rho_g+\lambda$ with $\lambda\in P^+$, the difference equation in Theorem \ref{HDE:thm} gives rise to a Pieri formula for the Jacobi polynomials by virtue of the specialization formula in Eq. \eqref{Hb}:
\begin{equation}\label{pieri}
E_{\omega}(x) P_{\lambda}(x)=
\sum_{\substack{\nu\in P(\omega)\\ \lambda+\nu\in P^+}}
\sum_{\eta\in W_\nu(w_\nu^{-1}\omega)}
 U_{\nu,\eta }(\rho_g +\lambda) V_{\nu }(\rho_g+\lambda ) P_{\lambda+\nu }(x) ,
\end{equation}
first for generic positive root multiplicity parameters $g_\alpha$ such that $\langle \rho_g,\alpha^\vee\rangle \neq 1$ ($\forall\alpha\in R$) and then for general positive root multiplicity values by continuity (cf. also Remark \ref{rm-reg} above).
The restriction of the first sum on the RHS of Eq. \eqref{pieri}---to those $\nu\in P(\omega)$ for which $\lambda+\nu$ is dominant---stems from the observation that the factor $V_\nu(\xi )$ vanishes at $\xi=\rho_g +\lambda$ if $\lambda+\nu\not\in P^+$.
Indeed, in such cases there exists a simple root $\beta$ such that $\langle\lambda+\nu,\beta^\vee\rangle <0$, i.e. either
$\langle\lambda ,\beta^\vee\rangle =0$ and $\langle\nu,\beta^\vee\rangle <0$ so $\langle\xi,\alpha^\vee\rangle +g_\alpha=0$ at $\xi=\rho_g+\lambda$
for $\alpha=-\beta$, or $\langle\lambda,\beta^\vee\rangle =1$ and $\langle\nu,\beta^\vee\rangle =-2$ so 
$1+\langle\xi,\alpha^\vee\rangle +g_\alpha=0$ at $\xi=\rho_g+\lambda$ for $\alpha =-\beta$. (Here we have exploited that $|\langle \nu,\alpha^\vee\rangle|\leq 2$ for all $\alpha\in R$ as well as the elementary fact that
$\langle\rho_g,\beta^\vee\rangle =g_\beta$ for $\beta$ simple.)

It is well-known that the Jacobi polynomials may be viewed as a $q\to 1$ degeneration of the more general Macdonald polynomials \cite[\S 11]{mac:orthogonal}. The Pieri formula in Eq. \eqref{pieri} arises in this context as the corresponding (rational) degeneration of the recently found Pieri formula for the Macdonald polynomials in \cite[Sec. 4]{die-ems:generalized}. Specifically, upon substituting
$t_\alpha=q_\alpha^{g_\alpha}$ in \cite[Thm. 4.1]{die-ems:generalized} and performing the limit $q\to 1$, Eq. \eqref{pieri} readily follows.

\subsection{Analytic continuation}
Our verification of the Pieri formula \eqref{pieri} proves Theorem \ref{HDE:thm} for $\xi\in \rho_g+P^+$. The extension to arbitrary spectral values $\xi\in V_\mathbb{C}$ is achieved through analytic continuation.
To this end we will follow a line of arguments that was recently employed successfully in the context of the Heckman-Opdam hypergeometric function by R\"osler et al, and which
hinges on the following
classical result from complex analysis due to Carlson, cf. \cite[Thm. 5.81]{tit:theory}.

\noindent {\bf Carlson's theorem.}
\emph{Let $f$ be an holomorphic function in a neighborhood of the closed right half-plane $$\mathbb{H}_r:=\{ z\in \mathbb{C}\mid \text{Re} (z)\geq 0\} ,$$ such that (i) $ f$  is bounded\footnote{In full generality Carlson's theorem allows for $f$ to have exponential growth of order $O(e^{c|z|})$ with $c<\pi$, but for our purposes the assumption of  a bounded function  suffices.}
on $\mathbb{H}_r$ and (ii) $f(m)=0$ for $m=0,1,2,\ldots$, then $f$ vanishes identically}.

See e.g. \cite[Sec. 4]{ros:positive}  for the proof of a product formula for the Heckman-Opdam hypergeometric function of type $BC$ based on Carlson's theorem and also \cite[Sec. 5]{ros-koo-voi:limit} for its use in the proof of a  limit transition between the type $BC$ and type $A$  Heckman-Opdam hypergeometric functions.
An earlier (and somewhat different) application of Carlson's theorem in the framework of the analysis on root systems can be found in Opdam's proof of the $q=1$ Macdonald constant term formulas (where it was used to perform analytic continuation with respect to the root multiplicity parameters) \cite{opd:some}.

Specifically, since the two sides of Eq. \eqref{HDEa} are $W$-invariant  and holomorphic both in $x\in U$ and in $\xi\in V_\mathbb{C}$, it is sufficient to prove the desired equality for
$x$ in the (closed) fundamental chamber $C:=\{x\in V \mid \langle x,\alpha^\vee\rangle\geq 0, \forall\alpha\in R^+\}$
and for
\begin{equation}\label{domain}
\xi=\rho_g+z_1\omega_1+\cdots+z_n\omega_n
\end{equation} with $z_j\in \mathbb{H}_r$ ($j=1,\ldots, n$). Here $\omega_1,\ldots ,\omega_n$ refers to the basis of the fundamental weights (so $n=\dim (V)=\text{rank}(R)$).
Upon division by $e^{\langle \xi ,x\rangle}$ and temporarily
assuming that $g_\alpha >1$ (so $\langle\rho_g,\alpha^\vee\rangle>1$ for all $\alpha\in R^+$),  both sides of the difference equation remain bounded in $z_j\in \mathbb{H}_r$, $j=1,\dots ,n$. 
Indeed, the estimate in  Eq. \eqref{Hc} reveals that for the variables of interest
$$ |e^{-\langle \xi ,x\rangle} F_{\xi +\nu} (x)| \leq    e^{\langle w_\nu(\nu)-\rho_g,x\rangle  } , $$
where we have used that
$\langle w\nu,x\rangle\leq \langle w_\nu (\nu),x\rangle$ and 
$\text{Re} \langle w(\xi-\rho_g),x\rangle \leq \text{Re}\langle \xi-\rho_g,x\rangle$ for all $w\in W$ (since both $x$ and $\text{Re} (\xi ) -\rho_g$ now belong to the fundamental chamber $C$).
Furthermore, the factors  of the form $\frac{\langle \xi ,\alpha^\vee\rangle + g_\alpha}{\langle \xi ,\alpha^\vee\rangle}$ and $\frac{\langle \xi ,\alpha^\vee\rangle +1\pm g_\alpha}{1+\langle \xi ,\alpha^\vee\rangle}$ constituting $V_\nu (\xi)$ \eqref{HDEb} and $U_{\nu,\eta} (\xi)$ \eqref{HDEc} remain bounded  on this variable domain as well (because for $g_\alpha >1$
 we stay away from the (affine) root hyperplanes where the denominators vanish and both quotients moreover tend to $1$ if $\langle\xi,\alpha^\vee\rangle \to\infty$).
When the coordinates $z_1,\ldots ,z_n$ all take nonnegative integral values the spectral variable $\xi$ \eqref{domain} belongs to the shifted dominant cone $\rho_g+P^+$ and the equality of both sides of Eq. \eqref{HDEa} is then guaranteed by the Pieri formula as argued above. We will now use Carlson's theorem to compare both sides of the difference equation outside this discrete spectral set. Indeed, successive extension of the  coordinate values $z_1,\ldots ,z_n$ to the complex half-space $\mathbb{H}_r$
by means of Carlson's theorem entails the desired equality for $\text{Re}(\xi) \in\rho_g + C$, $x\in C$ and $g_\alpha >1$.
As anticipated, the extension of our difference equation to  the full domain $\xi\in V_\mathbb{C}$, $x\in U$ and $g_\alpha >0$ is now immediate from the Weyl group symmetry  and the analyticity (cf. also Remark \ref{rm-reg} above).

\section{Hypergeometric difference  equations \\ for nonreduced root systems}\label{sec5}
In the standard euclidean realization of the nonreduced root system of rank $n$ \cite{bou:groupes}, the
hypergeometric differential equation in Eq.  \eqref{HDE} becomes of the form
\begin{subequations}
\begin{equation}
L F_{\xi} =(\xi_1^2+\cdots +\xi_n^2 - \rho_1^2-\cdots -\rho_n^2)F_{\xi} ,
\end{equation}
where $\rho_j=(n-j)g+\frac{1}{2}g_1+g_2$ ($j=1,\ldots,n$) and
\begin{align}
L&= \sum_{1\leq j\leq n} \left[ \frac{\partial^2}{\partial x_j^2}+
 \Bigl(g_1 \coth \frac{1}{2}  (x_j)  + 2g_2 \coth  (x_j)\Bigr) \frac{\partial}{\partial x_j} \right] +\\
&g \sum_{1\leq j<k\leq n}  \left[  \coth \frac{1}{2}  (x_j+x_k)  \Bigl( \frac{\partial}{\partial x_j}+ \frac{\partial}{\partial x_j}\Bigr)+\coth \frac{1}{2}  (x_j-x_k)  \Bigl( \frac{\partial}{\partial x_j}- \frac{\partial}{\partial x_j}\Bigr) \right] \nonumber .
\end{align}
\end{subequations}
Here $x=(x_1,\ldots ,x_n)$, $\xi=(\xi_1,\ldots ,\xi_n)$  and  $\rho_g=(\rho_1,\ldots ,\rho_n)$ are being represented in the standard orthogonal basis $e_1,\ldots ,e_n$ of $V$.
Rather than to incorporate the nonreduced setting  into Theorem \ref{HDE:thm}, it is actually more convenient to tweak the structure of our difference equation a bit in this case, as it will give rise to somewhat simpler formulas. 

\begin{theorem}
For $R$ nonreduced of rank $n$ and $\ell\in \{ 1,\ldots ,n\}$, the Heckman-Opdam hypergeometric function $F_\xi(x)$ satisfies the difference equation
\begin{subequations}
{\small
\begin{equation}\label{HDEnra}
\sum_{\substack{J\subset \{ 1,\ldots ,n\} ,\, 0\leq|J|\leq \ell\\
               \varepsilon_j=\pm 1,\; j\in J}}
\!\!\!\!\!\!\!\!\!
U_{J^c,\, \ell -|J|}(\xi)
V_{\varepsilon J}(\xi)
F_{\xi +e_{\varepsilon J}} (x) =E_\ell (x) F_{\xi} (x)
\end{equation}}
(as a holomorphic identity in $x\in U$ and $\xi \in V_\mathbb{C}$), with
{\small
\begin{align}
V_{\varepsilon J}(\xi )&:=
\prod_{j\in J} 
\frac{(\varepsilon_j\xi_j+\frac{1}{2}g_1+g_2)(1+2\varepsilon_j\xi_j+g_1)}{\varepsilon_j\xi_j(1+2\varepsilon_j\xi_j)}
\prod_{\substack{j\in J\\ k\not\in J}} 
\Bigl(\frac{\varepsilon_j\xi_j+\xi_{k}+g}{\varepsilon_j\xi_j+\xi_{k}}\Bigr)\Bigl(\frac{\varepsilon_j\xi_j-\xi_{k}+g}{\varepsilon_j\xi_j-\xi_{k}}\Bigr)
\nonumber \\
&  \times
\prod_{\substack{j,j^\prime \in J\\ j<j^\prime}}
\Bigl(\frac{\varepsilon_j\xi_j+\varepsilon_{j^\prime}\xi_{j^\prime}+g}{\varepsilon_j\xi_j+\varepsilon_{j^\prime}\xi_{j^\prime}}\Bigr)
\Bigl(\frac{1+\varepsilon_j\xi_j+\varepsilon_{j^\prime}\xi_{j^\prime}+g}{1+\varepsilon_j\xi_j+\varepsilon_{j^\prime}\xi_{j^\prime}}\Bigr) ,
\end{align}
\begin{align}
U_{K,p}(\xi):=
 (-1)^p
\sum_{\stackrel{I\subset K,\, |I|=p}
               {\varepsilon_i =\pm 1,\; i\in I }}
&\Biggl( \prod_{i\in I} 
\frac{(\varepsilon_i\xi_i+\frac{1}{2}g_1+g_2)(1+2\varepsilon_i\xi_i+g_1)}{\varepsilon_i\xi_i(1+2\varepsilon_i\xi_i)} \nonumber \\
&\times \prod_{\substack{i\in I\\ k\in K\setminus I}} 
\Bigl(\frac{\varepsilon_i\xi_i+\xi_{k}+g}{\varepsilon_i\xi_i+\xi_{k}}\Bigr)\Bigl(\frac{\varepsilon_i\xi_i-\xi_{k}+g}{\varepsilon_i\xi_i-\xi_{k}}\Bigr)
\nonumber \\
&  \times
\prod_{\substack{i,i^\prime \in I\\ i<i^\prime}}
\Bigl(\frac{\varepsilon_i\xi_i+\varepsilon_{i^\prime}\xi_{i^\prime}+g}{\varepsilon_i\xi_i+\varepsilon_{i^\prime}\xi_{i^\prime}}\Bigr)
\Bigl(\frac{1+\varepsilon_i\xi_i+\varepsilon_{i^\prime}\xi_{i^\prime}-g}{1+\varepsilon_i\xi_i+\varepsilon_{i^\prime}\xi_{i^\prime}}\Bigr) \Biggr) ,\end{align}}
 and
{\small
\begin{equation}\label{HDEnrd}
E_\ell (x):=4^{\ell }\sum_{\substack{J\subset \{ 1,\ldots, n\} \\ |J|=\ell}}   \prod_{j\in J}  \sinh^2\left(\frac{x_j}{2}\right) ,
\end{equation}}
\end{subequations}
where $e_{\varepsilon J} := \sum_{j\in J} \varepsilon_j e_j $, $|J|$ denotes the cardinality of $J\subset\{ 1,\ldots, n\}$, and $J^c:=\{ 1,\ldots, n\}\setminus J$.
\end{theorem}

\begin{proof}
The theorem follows by a straightforward modification of the three-stage proof in Section \ref{sec4}.
(i) The residue analysis to infer that the LHS of our difference equation is holomorphic in $\xi$ was carried out (at the trigonometric level) in \text{Ref.}
\cite{die:diagonalization}.
(ii) For $\xi_j=\rho_j+\lambda_j$ ($j=1,\ldots, n$),  with $\lambda_1,\ldots,\lambda_n$ being nonnegative integers such that $\lambda_1\geq\lambda_2\geq \cdots\geq\lambda_n\geq 0$, the stated difference equation reduces to a Pieri formula for the hyperoctahedral-symmetric Heckman-Opdam Jacobi polynomials that was derived in
 \cite[Thm 6.4]{die:properties}. 
 (iii) The extrapolation to general spectral values
 hinges again on Carlson's theorem, following closely the arguments in Section \ref{sec4}.
 \end{proof}

By varying $\ell$ from $1,\ldots ,n$, Eqs. \eqref{HDEnra}--\eqref{HDEnrd}
reveal that the Heckman-Opdam hypergeometric function associated with a nonreduced root system constitutes a joint eigenfunction
for the commuting quantum integrals of a rational Ruijsenaars-Schneider type system with hyperoctahedral symmetry introduced in
\cite{die:integrability} (see also \cite{die:difference}). 
For $\ell =1$, the difference equation at issue becomes
\begin{subequations}
\begin{align}\label{J1a}
\sum_{1\leq j\leq n}  V_j(\xi) \Bigl( F_{\xi +e_j}(x)-F_\xi (x)\Bigl)+V_j(-\xi) \Bigl( F_{\xi -e_j}(x)-F_\xi (x)\Bigl) & \\
 = 4\Bigl( \sinh^2\left(\frac{x_1}{2}\right)+\cdots +\sinh^2\left(\frac{x_n}{2}\right)  \Bigr) F_\xi (x) ,& \nonumber
\end{align}
where
\begin{equation}\label{J1b}
V_j(\xi)=\frac{(\xi_j+\frac{1}{2}g_1+g_2)(1+2\xi_j+g_1)}{\xi_j(1+2\xi_j)}
\prod_{\substack{1\leq k\leq n\\ k\neq j}}\Bigl(\frac{\xi_j+\xi_{k}+g}{\xi_j+\xi_{k}}\Bigr)\Bigl(\frac{\xi_j-\xi_{k}+g}{\xi_j-\xi_{k}}\Bigr) .
\end{equation}
\end{subequations}
This particular difference equation for $F_\xi(x)$ when $R$ is nonreduced was first established by Chalykh with the aid of a Baker-Akhiezer function associated with the root system of interest
\cite[Thm. 6.12]{cha:bispectrality}.

\begin{remark}
For $n=1$, the difference equation in Eqs. \eqref{J1a}, \eqref{J1b} boils down to the following elementary difference equation
\begin{subequations}
\begin{align}
&\frac{(\xi+\frac{1}{2}g_1+g_2)(1+2\xi+g_1)}{\xi (1+2\xi)} \Bigl( F_{\xi +1}(x) - F_\xi (x)\Bigr) +\nonumber\\
&\frac{(\xi-\frac{1}{2}g_1-g_2)(-1+2\xi-g_1)}{\xi (-1+2\xi)} \Bigl( F_{\xi -1}(x)-F_\xi (x)\Bigr) =
4 \sinh^2\left(\frac{x}{2}\right)  F_\xi (x) \label{de}
\end{align}
for the Gauss hypergeometric function (a.k.a. Jacobi function)
\begin{equation} \label{gf}
F_\xi (x) =  {}_2F_1
\left(
\begin{matrix}
-\xi +\frac{g_1}{2}+g_2, \, \xi +\frac{g_1}{2}+g_2\\ 
\frac{1}{2}+g_1+g_2 \end{matrix} \; ; \; -\sinh^2 \left(\frac{ x}{2}\right) \right) ,
\end{equation}
\end{subequations}
cf. \cite[p. 185]{rui:finite-dimensional} and \cite[Sec. 2]{rui:relativistic}. Notice that for $\xi=\frac{g_1}{2}+g_2+l$ ($l =0,1,2,\ldots$), this recovers precisely the well-known three-term recurrence relation
\begin{subequations}
\begin{align}
 \sinh^2 \left( \frac{x}{2}\right) \: P_l (x) = &
  \frac{(l+g_1+2g_2) (\frac{1}{2}+l+g_1+g_2)}
     {(2l+g_1+2g_2) (1+2l+g_1+2g_2)}
\left( P_{l+1}(x)-P_l(x) \right)  \nonumber \\
 + &\frac{l (-\frac{1}{2}+l+g_2)}
        {(2l+g_1+2g_2) (-1+2l+g_1+2g_2)}
\left( P_{l-1}(x)-P_l(x) \right) \label{rr}
\end{align}
for the normalized Jacobi polynomials
\begin{equation}\label{jp}
P_l(x) =
{}_2F_1
\left(
\begin{matrix}
-l,\; l+g_1+2g_2 \\ 
\frac{1}{2}+g_1+g_2 \end{matrix} \; ; \; -\sinh^2 \left( \frac{x}{2}\right) \right) \qquad (l=0,1,2,\ldots ),
\end{equation}
\end{subequations}
cf. e.g. \cite[Ch. 9.8]{koe-les-swa:hypergeometric}. In this rank-one situation, Carlson's theorem thus tells us that---reversely---the difference equation \eqref{de} for the Gauss hypergeometric function \eqref{gf} can be retrieved by analytic continuation from the recurrence relation
\eqref{rr} for the Jacobi polynomials \eqref{jp}. Indeed, both sides of Eq. \eqref{de} are manifestly entire in $\xi$ and holomorphic in $g_1$, $g_2$ provided $\frac{1}{2}+g_1+g_2\not\in \{ 0,-1,-2,\ldots\}$, moreover,
upon division by $e^{\xi |x|}$, the expressions
remain bounded for
$\xi=\frac{g_1}{2}+g_2+z$ on the domain $\text{Re}(z)\geq 0$ if $\frac{g_1}{2}+g_2>\frac{1}{2}$ (cf. \cite[Lem. 11]{fle:paley}).
\end{remark}

\section{Confluent hypergeometric limit of Toda-Whittaker type}\label{sec6}
Let us return to our main setting of a reduced root system $R$. The class-one Whittaker function
$\bar{F}_\xi(x)$ \cite{jac:fonctions,kos:quantization,has:whittaker,bau-oco:exponential} is a confluent hypergeometric function
diagonalizing the quantum Hamiltonian of the open Toda chain
\begin{equation}\label{toda-ep}
\bar{L} \bar{F}_\xi(x) =\langle \xi,\xi\rangle \bar{F}_\xi(x)\quad\text{with}\quad \bar{L}:=\Delta -2 \sum_{\alpha\in S} e^{-\alpha} ,
\end{equation}
where $S\subset R^+$ denotes the basis of the simple roots. The function in question is real-analytic and of moderate exponential growth in $x\in V$, and it is holomorphic in the spectral variable $\xi\in V_\mathbb{C}$.
Following \cite[Sec. 4]{bau-oco:exponential}, we will normalize our Whittaker function such that it is $W$-invariant in the spectral variable
\begin{subequations}
\begin{equation}\label{winv}
\bar{F}_{w\xi}(x)=\bar{F}_\xi (x)\quad (\forall w\in W),
\end{equation}
and characterized by an asymptotics 
for $\text{Re}(\xi)$ in the open fundamental chamber $\text{Int}(C)=\{ x \in V \mid \langle x,\alpha^\vee\rangle >0,\forall\alpha\in R^+\}$ of the form
\begin{equation}\label{wasp}
\lim_{x\to + \infty}  e^{\langle w_0 \xi,x\rangle} \bar{F}_\xi (x)= \prod_{\alpha\in R^+}  \eta_\alpha^{-\langle \xi,\alpha^\vee\rangle } \, \Gamma (\langle \xi,\alpha^\vee\rangle ) \quad\text{with}\quad \eta_\alpha := \sqrt{ \frac{2}{\langle \alpha ,\alpha\rangle}}.
\end{equation}
\end{subequations}
Here $w_0$ refers to the longest element of $W$, and
by $x\to +\infty$ it is meant that $\langle x,\alpha^\vee\rangle\to +\infty$ for all $\alpha\in R^+$.

In \cite[Thm. 3]{shi:limit} it was shown that, 
for $\xi\in V_\mathbb{C}$ generic such that
\begin{equation}\label{generic}
 \langle 2\xi+\nu,\nu\rangle \neq 0 \qquad \forall \nu\in Q\setminus \{0\} 
\end{equation}
and upon parametrizing the root multiplicities $g_\alpha$ in terms of $t\in \mathbb{R}$ such that
\begin{equation}\label{par}
g_\alpha (t) (g_\alpha (t)-1)=\eta_\alpha^2 e^{t},
\end{equation}
one recovers the Whittaker function $\bar{F}_\xi (x)$ with $x\in \text{Int}(C)$ as a confluent limit of a suitably dressed Heckman-Opdam hypergeometric function (diagonalizing the hyperbolic Calogero-Moser Hamiltonian):
\begin{subequations}
\begin{equation}\label{lim1}
\bar{F}_\xi (x)=  \lim_{t\to +\infty}     N_t\,   \delta_t\, (x+t\rho^\vee)     F_{\xi } (x+t\rho^\vee ;t)\qquad (x\in \text{Int}(C)),
\end{equation}
where $F_{\xi} (x;t ):= F_\xi (x)$ with $g_\alpha=g_\alpha (t)$ given by Eq. \eqref{par}, 
\begin{equation}\label{lim2}
N_t:=\prod_{\alpha\in R^+}  \frac{\Gamma (\langle \rho_{g(t)},\alpha^\vee\rangle )\Gamma (g_\alpha (t) )} {\Gamma (\langle \rho_{g(t)},\alpha^\vee\rangle+g_\alpha (t))},\qquad
\delta_t := \prod_{\alpha\in R^+}  (e^{\alpha /2}-e^{-\alpha /2})^{g_\alpha (t)} ,
\end{equation}
\end{subequations}
and $\rho^\vee:=\frac{1}{2}\sum_{\alpha\in R^+} \alpha^\vee$.
The proof of this limiting relation in \cite{shi:limit}
is based on the connection formulas (i.e. $c$-function expansions) for $F_\xi (x)$ and $\bar{F}_\xi (x)$ \cite{hec-sch:harmonic,has:whittaker}.  
Indeed, first it is shown that a (dressed)
Harish-Chandra series solution of the
hypergeometric differential equation converges to that of the quantum Toda chain (i.e. to the fundamental Whittaker function).
The limiting relation then readily follows by comparing the explicit formulas for the respective $c$-functions in terms of gamma functions.

The combination of Theorem \ref{HDE:thm} and the limit in Eqs. \eqref{lim1}, \eqref{lim2} entails the following confluent hypergeometric difference equation for the class-one Whittaker function.
\begin{theorem}\label{CHDE:thm}
Let $R$ be reduced and $\omega\in P^+$ small such that $\langle \omega,\alpha^\vee\rangle\leq 2$ for all $\alpha\in R^+$.
The class-one Whittaker function $\bar{F}_\xi(x)$, determined by the quantum Toda eigenvalue equation \eqref{toda-ep} and the normalization \eqref{winv}, \eqref{wasp}, satisfies the difference equation
\begin{subequations}
\begin{equation}\label{CHDEa}
\sum_{\nu\in P(\omega)}
\sum_{\eta\in W_\nu(w_\nu^{-1}\omega)}
 \bar{U}_{\nu,\eta }(\xi) \bar{V}_{\nu }(\xi) \bar{F}_{\xi+\nu }(x) =\bar{E}_{\omega}(x) \bar{F}_{\xi}(x)
\end{equation}
(as a holomorphic identity in $\xi\in V_\mathbb{C}$ with $x\in V$), where
\begin{eqnarray}\label{CHDEb}
&& \bar{V}_{\nu }(\xi ):=\prod_{\substack{\alpha\in R\\ \langle\nu,\alpha^\vee \rangle>0}} \frac{\eta_\alpha}{\langle \xi, \alpha^\vee \rangle}
      \prod_{\substack{\alpha\in R\\ \langle\nu, \alpha^\vee \rangle=2}} \frac{\eta_\alpha }{1+ \langle \xi,\alpha^\vee \rangle} ,\\
&& \bar{U}_{\nu,\eta } (\xi ):=\prod_{\substack{\alpha\in R_\nu \\ \langle\eta,\alpha^\vee\rangle>0}}
\frac{\eta_\alpha}{\langle \xi,\alpha^\vee  \rangle}
\prod_{\substack{\alpha\in R_\nu \\ \langle\eta,\alpha^\vee\rangle=2}} 
 \frac{-\eta_\alpha }{1+ \langle \xi, \alpha^\vee  \rangle} , \label{CHDEc}
\end{eqnarray}
and
\begin{equation}\label{CHDEd}
\bar{E}_{\omega }(x):=e^\omega (x)=e^{\langle \omega ,x\rangle} .
\end{equation}
\end{subequations}
\end{theorem}
\begin{proof}
Let us temporarily assume that $x$ belongs to the open fundamental chamber $\text{Int}(C)$ and that the spectral parameter $\xi\in V_\mathbb{C}$ is generic in the sense that it satisfies the inequalities \eqref{generic} and does not belong to the weight lattice $P$.
Starting from the difference equation in Theorem \ref{HDE:thm}, we
parametrize the root multiplicities in accordance with Eq. \eqref{par}, multiply both sides by $e^{-t\langle \omega ,\rho^\vee\rangle} N_t\delta_t (x)$, and replace $x$ by $x+t\rho^\vee$.
The difference equation in Eqs. \eqref{CHDEa}--\eqref{CHDEd} is now recovered for $t\to+\infty$, in view of Eq. \eqref{lim1} and the limits
$$
\lim_{t\to +\infty} e^{-t\langle \omega ,\rho^\vee\rangle}E_\omega (x+t\rho^\vee) = \bar{E}_\omega (x) ,
$$
$$
\lim_{t\to +\infty} e^{-t\langle w_\nu \nu ,\rho^\vee\rangle} V_\nu  (\xi ) = \bar{V}_\nu (\xi ) ,
$$
and
$$
\lim_{t\to +\infty} e^{-t\langle \omega-w_\nu \nu ,\rho^\vee\rangle} U_{\nu ,\eta} (\xi ) = \bar{U}_{\nu ,\eta} (\xi ) .
$$
The first two of these  limits are evident, while for the last limit we used the property in Remark \ref{homogeneity} below (with $g_\alpha=1$).
Finally, the domain restrictions on $x$ and our genericity assumptions
on the spectral parameter $\xi$ are readily removed by analytic continuation, since $ \bar{F}_{\xi}(x)$ is real-analytic in $x\in V$ and both sides of the stated
difference equation extend holomorphically to $\xi\in V_\mathbb{C}$ (cf. Section \ref{holomorphy:sec}).
\end{proof}
For $\omega$ minuscule the difference equation in Theorem \ref{CHDE:thm} becomes of the form
\begin{equation}\label{Cmin}
\sum_{\nu\in W\omega} \bar{V}_{\nu }(\xi) \bar{F}_{\xi+\nu }(x) = e^{\langle \omega ,x\rangle}  \bar{F}_{\xi}(x),
\end{equation}
whereas for $\omega$ quasi-minuscule it reads
\begin{equation}\label{Cqmin}
\sum_{\nu\in W\omega}  \bar{V}_{\nu }(\xi) \bigl( \bar{F}_{\xi+\nu }(x) -\bar{F}_{\xi}(x) \bigr) = e^{\langle \omega ,x\rangle}  \bar{F}_{\xi}(x).
\end{equation}
When $R$ is of type $A$ one recovers from Eq. \eqref{Cmin} the corresponding difference equations for the Whittaker function found in Refs. \cite{kar-leb:integral,bab:equations,bor-cor:macdonald,skl:bispectrality,koz:aspects}, by varying $\omega$ over the fundamental weights (which in this special case are off course all minuscule).

\begin{remark}\label{homogeneity} 
To determine the asymptotic exponential growth-rate of the factor $U_{\nu, \eta} (\xi)$ for $t\to +\infty$ in the proof of Theorem \ref{CHDE:thm},
we exploited the property that for
any $\mu,\omega \in P^+$ with $\mu<\omega$ and $\langle\omega,\alpha^\vee\rangle \leq 2$ ($\forall\alpha\in R^+$):
$$
\sum_{\substack{\alpha\in R^+\\ \langle\mu,\alpha^\vee\rangle >0}} g_\alpha \langle\mu,\alpha^\vee\rangle 
=
\sum_{\substack{\alpha\in R^+\\ \langle\mu,\alpha^\vee\rangle >0}} g_\alpha \langle\omega,\alpha^\vee\rangle $$
(as is readily verified on a case-by-case basis for each type of irreducible root system).
\end{remark}

\bibliographystyle{amsplain}

\end{document}